\documentclass[12pt]{amsart}
\usepackage[english]{babel}
\usepackage[utf8]{inputenc}
\usepackage{lipsum}
\usepackage{mathrsfs}
\usepackage{fullpage}
\usepackage{mathtools}
\usepackage{amsmath}
\usepackage[colorlinks]{hyperref}
\usepackage{fullpage}
\usepackage{mathtools}
\usepackage{amsmath}
\numberwithin{equation}{section}
\theoremstyle{plain}
\newtheorem{theorem}{Theorem}[section]

\newtheorem{cor}[theorem]{Corollary}
\newtheorem{definition}[theorem]{Definition}
\newtheorem{lemma}[theorem]{Lemma}

\newcommand{\D}{\mathcal D}
\begin{document}

\title[Rayleigh-Faber-Krahn, Lyapunov and Hartmann-Wintner inequalities]{Rayleigh-Faber-Krahn,  Lyapunov and Hartmann-Wintner inequalities for fractional elliptic problems}

\author[A. Kassymov]{Aidyn Kassymov}
\address{
   Aidyn Kassymov:
  \endgraf
  Department of Mathematics: Analysis, Logic and Discrete Mathematics
  \endgraf
  Ghent University, Belgium
  \endgraf
   and
  \endgraf
  Al-Farabi Kazakh National University
  \endgraf
  Almaty, Kazakhstan
  \endgraf
  and
  \endgraf
  Institute of Mathematics and Mathematical Modeling
  \endgraf
  Almaty, Kazakhstan
  \endgraf
  {\it E-mail address} {\rm kassymov@math.kz} and {\rm aidyn.kassymov@ugent.be}
}

\author[Michael Ruzhansky]{Michael Ruzhansky}
\address{
  Michael Ruzhansky:
  \endgraf
Department of Mathematics: Analysis,
Logic and Discrete Mathematics
  \endgraf
Ghent University, Belgium
  \endgraf
 and
  \endgraf
 School of Mathematical Sciences
 \endgraf
Queen Mary University of London
\endgraf
United Kingdom
\endgraf
  {\it E-mail address} {\rm michael.ruzhansky@ugent.be}
 }

\author[B. T. Torebek]{Berikbol T. Torebek}
\address{
  Berikbol T. Torebek:
    \endgraf
  Department of Mathematics: Analysis, Logic and Discrete Mathematics
  \endgraf
  Ghent University, Belgium
  \endgraf
   and
  \endgraf
  Institute of Mathematics and Mathematical Modeling
  \endgraf
  Almaty, Kazakhstan
  \endgraf
   and
  \endgraf
  Al-Farabi Kazakh National University
  \endgraf
  Almaty, Kazakhstan
  \endgraf
  {\it E-mail address} {\rm berikbol.torebek@ugent.be}
  }

\thanks{The authors were supported in parts by the FWO Odysseus 1 grant G.0H94.18N: Analysis and Partial Differential Equations. The second author was supported by EPSRC grant EP/R003025/1 and by the Leverhulme Grant RPG-2017-151. The third author was supported by a grant
No.AP08052046 from the Ministry of Science and Education of the Republic of Kazakhstan}


\begin{abstract}
In this paper in the cylindrical domain we consider a fractional elliptic operator with Dirichlet conditions.
We prove, that the first eigenvalue of the fractional elliptic operator is minimised in a circular cylinder among all cylindrical domains of the same Lebesgue measure.  This inequality is called the Rayleigh-Faber-Krahn inequality.
Also, we
give Lyapunov and Hartmann-Wintner inequalities for the fractional elliptic boundary value
problem.
\end{abstract}

\maketitle
\tableofcontents
\section{Introduction}
Let $\Omega\subset \mathbb{R}^N,\,N>2,$ be an open bounded domain with smooth boundary and $(a,b),\,-\infty<a<b<+\infty$ be an interval. In cylindrical domain $D=(a,b)\times\Omega$ we define the operator
\begin{equation}\label{FrLap}
\mathcal{L}^{\alpha,s}u(x,y):\equiv D_{a+,x}^{\alpha} \mathcal{D}_{b-,x}^{\alpha}u(x,y)+(-\Delta)^{s}_{y}u(x,y),\,\,\,(x,y)\in D,\end{equation} with Dirichlet boundary conditions
\begin{equation}\label{DirCon1}u(a,y)=u(b,y)=0,\,\,\,\,y\in \Omega,\end{equation}
\begin{equation}\label{DirCon2}u(x,y)=0\,\,\,\,y\in \mathbb{R}^{N}\setminus \Omega,\end{equation}
where $1/2<\alpha\leq 1$ and $s\in(0,1).$ Here $D_{a+,x}^\alpha
u\left(x,y\right) = \partial_xI_{a+,x}^{1-\alpha} u(x,y)$ is the left Riemann-Liouville, $
\mathcal{D}_{b-,x}^\alpha
u\left(x,y\right) = I_{b-,x} ^{1-\alpha} u_x(x,y)$ is the right Caputo fractional derivatives of order $0<\alpha\leq 1$ with  the left and right Riemann–Liouville fractional integrals
$$
I_{a+,x}^\alpha u\left(x,y\right) = \frac{1}{{\Gamma \left( \alpha \right)}}\int\limits_a^x {\left(
{x - s} \right)^{\alpha  - 1} u\left( s,y \right)} ds$$ and $$I_{b-,x}^\alpha u\left( x,y \right) = \frac{1}{{\Gamma \left(\alpha \right)}}\int\limits_x^b {\left({s - x} \right)^{\alpha  - 1} u\left( s,y \right)} ds,\, x\in(a,b),
$$
respectively, and $(-\Delta)^{s}$ is the fractional Laplacian of order $s\in(0,1)$ defined by \begin{equation}
(-\Delta)^{s}_yu(x,y)=C_{N,s}\int_{\mathbb{R}^{N}} \frac{u(x,y)-u(x,\xi)}{|y-\xi|^{N+2s}}d\xi, \,\,\,y\in\mathbb{R}^{N},\end{equation}
where $C_{N,s}$ is some normalisation constant.

The main goals of this paper are to obtain the Rayleigh-Faber-Krahn and Lyapunov  inequalities for the boundary value problem \eqref{FrLap}, \eqref{DirCon1}, \eqref{DirCon2}.

It is known that the first eigenvalue of the multidimensional Dirichlet-Laplacian is minimised in a ball among all domains of the same Lebesgue measure.  This inequality is called the Rayleigh-Faber-Krahn inequality \cite{Hen}. Recently, some Rayleigh-Faber-Krahn type inequalities were obtained for the volume potentials \cite{Ruz, Ruz1, RuzSS}, and for the fractional elliptic operators \cite{BrPa, C17}.

For the second order differential equation with Dirichlet boundary condition Lyapunov \cite{Lyap} proved a necessary condition of existence of non-trivial solutions. In \cite{Hart}, Hartman and Wintner generalised the Lyapunov inequality.  Recently, the study of Lyapunov-type inequalities was extended to the multidimensional elliptic problems by some authors \cite{Nap, Kir1, Kir2, Odz}.

For the convenience of the reader, let us briefly summarise the results of this paper:
\begin{itemize}
\item {\bf Rayleigh-Faber-Krahn inequality for circular cylinder.} Suppose that $\frac{1}{2}<\alpha\leq 1$ and $s\in(0,1)$. Then the first eigenvalue of the problem \begin{equation*}
\mathcal{L}^{\alpha,s}u(x,y)=\nu u(x,y),\,\,\textrm{in}\,\,D=(a,b)\times\Omega,
\end{equation*} with boundary conditions \eqref{DirCon1}-\eqref{DirCon2} is minimised in the circular cylinder $\mathcal{C}$ among all cylindric domains of a given measure, that is
\begin{equation*}
   \nu_{1}(D)\geq\nu_{1}(\mathcal{C}),
\end{equation*}
for all $D$  with $|D| = |\mathcal{C}|$.
\item {\bf Rayleigh-Faber-Krahn inequality for polygonal cylinder.} Suppose that $\frac{1}{2}<\alpha\leq 1$ and $s\in(0,1)$. Then first eigenvalue of the problem \begin{equation*}
\mathcal{L}^{\alpha,s}u(x,y)=\nu u(x,y),\,\,\textrm{in}\,\,D=(a,b)\times\Omega,
\end{equation*} with boundary conditions \eqref{DirCon1}-\eqref{DirCon2} is minimised in the equilateral triangular (or square) cylinder $D^{\star}=(a,b)\times\Omega^{\star}$ among all triangular (or quadrilateral) cylindric domains of a given measure, that is
\begin{equation}
   \nu_{1}(D)\geq\nu_{1}(D^{\star}),
\end{equation}
for all $D$  with $|D| = |D^{\star}|$.
\item {\bf Lyapunov inequality.} Assume that $\frac{1}{2}<\alpha\leq 1,$ $s\in(0,1)$ and $q\in C([a,b])$. Then for the fractional elliptic equation
 \begin{equation*}
\mathcal{L}^{\alpha,s}u(x,y)=q(x)u(x,y),\,\,\textrm{in} \,\,D=(a,b)\times\Omega,
\end{equation*} with boundary conditions \eqref{DirCon1}-\eqref{DirCon2} we have
\begin{equation*}
\int_{a}^{b}|q(x)-\lambda_{1}(\Omega)|dx\geq\left(\sup\limits_{a<x<b}G(x,x)\right)^{-1},\end{equation*}
where $\lambda_{1}(\Omega)$ is the first eigenvalue of the fractional Dirichlet-Laplacian \eqref{spec} and $G(x,t)=K(x,t)-\frac{K(a,t)K(x,a)}{K(a,a)}$ and $K(x,t)=\frac{1}{\Gamma^{2}(\alpha)}\int_{\max\{x,t\}}^{b}(s-x)^{\alpha-1}(s-t)^{\alpha-1}ds.$
\item {\bf Hartmann-Wintner inequality.} Assume that $\frac{1}{2}<\alpha\leq 1,$ $s\in(0,1)$ and $q\in C([a,b])$. Then for the fractional elliptic equation
 \begin{equation*}
\mathcal{L}^{\alpha,s}u(x,y)=q(x)u(x,y),\,\,\textrm{in} \,\,D=(a,b)\times\Omega,
\end{equation*} with boundary conditions \eqref{DirCon1}-\eqref{DirCon2} we have
\begin{equation*}\int\limits_a^b \left(K(a,a)K(s,s)-K^{2}(a,s)\right) [q(x)-\lambda_{1}(\Omega)]^+ ds> \frac{(b-a)^{2\alpha-1}}{\Gamma^{2}(\alpha)(2\alpha-1)},
\end{equation*}
where $[q(x)-\lambda_{1}(\Omega)]^+ = \max\{q(x)-\lambda_{1}(\Omega), 0\}.$
\end{itemize}

\section{One dimensional fractional boundary value problem}
Let us consider the following problem:
\begin{equation}\label{mainproblem}
    \begin{cases}
D^{\alpha}_{a+}\D^{\alpha}_{b-}u(x)-q(x)u(x)=0,\,\,\,\,x\in(a,b),\,\,\,\alpha\in\left(\frac{1}{2},1\right],\\
u(a)=u(b)=0.
\end{cases}
\end{equation}
\begin{theorem}\label{thm1}
Assume that $\alpha\in\left(\frac{1}{2},1\right]$ and let $u=u(x)$ be the solution of \eqref{mainproblem}. Then the solution of \eqref{mainproblem} is the solution of the following integral equation
\begin{equation}\label{Int}    u(x)=\int_{a}^{b}G(x,t)q(t)u(t)dt,
\end{equation}
where $$G(x,t)=K(x,t)-\frac{K(a,t)K(x,a)}{K(a,a)}$$ and $$K(x,t)=\frac{1}{\Gamma^{2}(\alpha)}\int_{\max\{x,t\}}^{b}(s-x)^{\alpha-1}(s-t)^{\alpha-1}ds.$$
\end{theorem}
\begin{proof}
By acting with the Riemann-Liouville fractional integral on problem \eqref{mainproblem}, and using the property \cite[Lemma 2.5]{1} $$I_{a+}^\alpha D_{a+}^\alpha u(x)= u(x)-\frac{I^{1-\alpha}_{a+}u(a)}{\Gamma(\alpha)}(x-a)^{\alpha-1}, \, 0<\alpha\leq 1,$$ we have
\begin{equation}
    \begin{split}
        0&=I^{\alpha}_{a+}D^{\alpha}_{a+}\D^{\alpha}_{b-}u(x)-I^{\alpha}_{a+}(q(x)u(x))\\&
        =\D^{\alpha}_{b-}u(x)-\frac{I^{1-\alpha}_{a+}\D^{\alpha}_{b-}u(a)}{\Gamma(\alpha)}(x-a)^{\alpha-1}-I^{\alpha}_{a+}(q(x)u(x)).
    \end{split}
\end{equation}
By acting with $I^{\alpha}_{b-},$ with $u(b)=0,$ and applying the property \cite[Lemma 2.22]{1} $$I_{b-}^\alpha \mathcal{D}_{b-}^\alpha u(x)=u(x)- u(b), \, 0<\alpha\leq 1,$$ we obtain
\begin{equation*}
    \begin{split}
        0&=I^{\alpha}_{b-}\D^{\alpha}_{b-}u(x)-\frac{I^{1-\alpha}_{a+}\D^{\alpha}_{b-}u(a)}{\Gamma(\alpha)}I^{\alpha}_{b-}(x-a)^{\alpha-1}-I^{\alpha}_{b-}I^{\alpha}_{a+}(q(x)u(x))\\&
        =u(x)-u(b)-\frac{I^{1-\alpha}_{a+}\D^{\alpha}_{b-}u(a)}{\Gamma(\alpha)}I^{\alpha}_{b-}(x-a)^{\alpha-1}-I^{\alpha}_{b-}I^{\alpha}_{a+}(q(x)u(x))\\&
        =u(x)-\frac{I^{1-\alpha}_{a+}\D^{\alpha}_{b-}u(a)}{\Gamma^2(\alpha)}\int_{x}^{b}(t-x)^{\alpha-1}(t-a)^{\alpha-1}dt
        \\&-\frac{1}{\Gamma^{2}(\alpha)}\int_{x}^{b}(s-x)^{\alpha-1}\left(\int_{a}^{s}(s-t)^{\alpha-1}q(t)u(t)dt\right)ds\\&
        =u(x)-\frac{I^{1-\alpha}_{a+}\D^{\alpha}_{b-}u(a)}{\Gamma^2(\alpha)}\int_{x}^{b}(t-x)^{\alpha-1}(t-a)^{\alpha-1}dt
        \\&-\frac{1}{\Gamma^{2}(\alpha)}\int_{a}^{b}q(t)u(t)\left(\int_{\max\{x,t\}}^{b}(s-x)^{\alpha-1}(s-t)^{\alpha-1}ds\right)dt\\&
 =u(x)-I^{1-\alpha}_{a+}\D^{\alpha}_{b-}u(a)K(x,a)-\int_{a}^{b}K(x,t)q(t)u(t)dt.
\end{split}
\end{equation*}
By the condition $u(a)=0$, we get
\begin{equation*}
    \begin{split}
        0&=u(a)-I^{1-\alpha}_{a+}\D^{\alpha}_{b-}u(a)K(a,a)-\int_{a}^{b}K(a,t)q(t)u(t)dt\\&
        =-I^{1-\alpha}_{a+}\D^{\alpha}_{b-}u(a)K(a,a)-\int_{a}^{b}K(a,t)q(t)u(t)dt,
\end{split}
\end{equation*}
then, we have
$$I^{1-\alpha}_{a+}\D^{\alpha}_{b-}u(a)=-\frac{1}{K(a,a)}\int_{a}^{b}K(a,t)q(t)u(t)dt.$$

Finally, we have
\begin{equation}
    \begin{split}
        u(x)=\int_{a}^{b}\left[K(x,t)-\frac{K(a,t)K(x,a)}{K(a,a)}\right]q(t)u(t)dt=\int_{a}^{b}G(x,t)q(t)u(t)dt.
    \end{split}
\end{equation}
The proof is complete.
\end{proof}
Next let us prove one of the main tools to show the Lyapunov inequality.
\begin{lemma}\label{lem1}
Assume that $\alpha\in\left(\frac{1}{2},1\right)$. Then we have
\begin{equation}\label{eq1}
\sup_{a<t<x}G(x,t)=G(x,x)>0,\,\,\,\,\,a<t<x<b,
\end{equation}
where $G(x,t)$ is defined in Theorem \ref{thm1}.
\end{lemma}
\begin{proof}
Firstly, we have
\begin{equation}\label{ten3}
    G(x,a)=K(x,a)-\frac{K(a,a)K(x,a)}{K(a,a)}=0.
\end{equation}
Then, we have
\begin{equation}\label{rav}
    \frac{\partial G(x,t)}{\partial t}=\frac{\partial K(x,t)}{\partial t}-\frac{ K(x,a)}{K(a,a)}\frac{\partial K(a,t)}{\partial t}.
\end{equation}
Let us calculate the first term of the right hand side of \eqref{rav}. For $a<t<x<b,$ we have
\begin{equation}\label{ten1}
    \begin{split}
        \frac{\partial K(x,t)}{\partial t}&=\frac{1}{\Gamma^{2}(\alpha)}\frac{\partial}{\partial t}\int_{\max\{x,t\}}^{b}(s-x)^{\alpha-1}(s-t)^{\alpha-1}ds\\&
        \stackrel{a<t<x<b}=\frac{1}{\Gamma^{2}(\alpha)}\frac{\partial}{\partial t}\int_{x}^{b}(s-x)^{\alpha-1}(s-t)^{\alpha-1}ds\\&
        =\frac{1-\alpha}{\Gamma^{2}(\alpha)}\int_{x}^{b}(s-x)^{\alpha-1}(s-t)^{\alpha-2}ds
        >0.
    \end{split}
\end{equation}
Then let us calculate the second term of \eqref{rav}. By integrating by parts, we get
\begin{equation}
    \begin{split}
        K(a,t)&=\frac{1}{\Gamma^{2}(\alpha)}\int_{\max\{a,t\}}^{b}(s-a)^{\alpha-1}(s-t)^{\alpha-1}ds\\&
       \stackrel{a=x<t}=\frac{1}{\Gamma^{2}(\alpha)}\int_{t}^{b}(s-a)^{\alpha-1}(s-t)^{\alpha-1}ds\\&
       =\frac{1}{\Gamma^{2}(\alpha)\alpha}(b-a)^{\alpha-1}(b-t)^{\alpha}-\frac{\alpha-1}{\Gamma^{2}(\alpha)\alpha}\int_{t}^{b}(s-a)^{\alpha-2}(s-t)^{\alpha}ds\\&
       =\frac{1}{\Gamma^{2}(\alpha)\alpha}(b-a)^{\alpha-1}(b-t)^{\alpha}+\frac{1-\alpha}{\Gamma^{2}(\alpha)\alpha}\int_{t}^{b}(s-a)^{\alpha-2}(s-t)^{\alpha}ds.
    \end{split}
\end{equation}
Hence, we have
\begin{equation}\label{ten2}
    \frac{\partial K(a,t)}{\partial t}=-\frac{1}{\Gamma^{2}(\alpha)}(b-a)^{\alpha-1}(b-t)^{\alpha-1}-\frac{1-\alpha}{\Gamma^{2}(\alpha)}\int_{t}^{b}(s-a)^{\alpha-2}(s-t)^{\alpha-1}ds<0.
\end{equation}
It easy to see that $K(x,a)>0$ for all $x\in[a,b]$. Then by combining this fact, \eqref{ten1} and \eqref{ten2}, we have
$$\frac{\partial G(x,t)}{\partial t}=\frac{\partial K(x,t)}{\partial t}-\frac{ K(x,a)}{K(a,a)}\frac{\partial K(a,t)}{\partial t}>0.$$
It means $G(x,t)$ is an increasing function in the variable $t$. By using this fact with \eqref{ten3}, we obtain
\begin{equation}
\sup_{a<t<x}G(x,t)=G(x,x)>0,\,\,\,\,\,a<t<x<b,
\end{equation} completing the proof.
\end{proof}

\section{Rayleigh-Faber-Krahn-type inequality on cylindrical domains}

Let us consider the following eigenvalue problem in cylindrical domain:
\begin{equation}\label{spec2}
\mathcal{L}^{\alpha,s}u(x,y)=\nu u(x,y),\,\,\textrm{in}\,\,D=(a,b)\times\Omega,
\end{equation} with Dirichlet boundary conditions \eqref{DirCon1}-\eqref{DirCon2},
where $\Omega$ is a bounded domain.

First, we give the following auxiliary statement for the application of further research.
\begin{lemma}\label{lm1}
Let $1/2<\alpha\leq 1.$ Then, the operator $D^{\alpha}_{a+}\D^{\alpha}_{b-}$ with Dirichlet boundary conditions:
\begin{itemize}
\item is self-adjoint and positive in $L^2([a,b]);$
\item the spectrum is discrete, positive and increasing.
\end{itemize}
\end{lemma}
\begin{proof}
It follows from Theorem \ref{thm1} that the inverse operator to $D^{\alpha}_{a+}\D^{\alpha}_{b-}$ with Dirichlet boundary conditions has the form \eqref{Int}. Then the symmetry and positivity of the kernel $G(x,t)$ implies the self-adjointness and positivity of the operator \eqref{Int} in $L^2([a,b]).$ Hence, all eigenvalues of operator $D^{\alpha}_{a+}\D^{\alpha}_{b-}$ with Dirichlet boundary conditions are real and positive. It is easy to show the complete continuity of the kernel $G(x,t)$, then the discreteness of the spectrum of the operator \eqref{Int} follows from this. This completes the proof.\end{proof}

\begin{theorem}
Let $1/2<\alpha\leq 1$ and $s\in(0,1).$ Then, the operator \eqref{FrLap}-\eqref{DirCon2} is self-adjoint and positive in $L^2(D),$ and all its eigenvalues are discrete, positive and increasing.
\end{theorem}
The theorem is proved by the method of separation of variables, by reducing problem \eqref{FrLap}-\eqref{DirCon2} to two self-adjoint operators: fractional Dirichlet-Laplacian and $D^{\alpha}_{a+}\D^{\alpha}_{b-}$ with Dirichlet boundary conditions.

From \cite{BrPa}, we can choose the first eigenfunction of
   \begin{equation}\label{spec}
 \begin{cases}
(-\Delta)_{y}^{s}\varphi_{1}(y)=\lambda_{1}(\Omega)\varphi_{1}(y),\,\,\,\,y\in \Omega,\\
   \varphi_{1}(y)=0,\,\,\,\,y\in \mathbb{R}^{N}\setminus \Omega,
 \end{cases}
\end{equation}
to be positive, corresponding to the simple and positive first eigenvalue $\lambda_{1}(\Omega)>0.$ It is known that the first eigenvalue $\lambda(\Omega)$ of the fractional Dirichlet-Laplacian \eqref{spec} is minimised in a ball $B$ among all domains of the same Lebesgue measure (see \cite[Th 3.5]{BrPa}), i.e.
\begin{equation}\label{FEi}
\lambda_1(\Omega)\geq \lambda_1(B),\,\,|\Omega|=|B|.
\end{equation}

Let us denote by $|\cdot|$ the Lebesgue measure, and let us introduce the Rayleigh quontient for the problem \eqref{spec2}, \eqref{DirCon1}-\eqref{DirCon2} in the following form:
\begin{equation}
    \nu_{1}(D)=\inf_{u\neq0}\frac{\langle \mathcal{L}^{\alpha,s}u,u\rangle}{\|u\|^{2}_{L^{2}(D)}},
\end{equation}
where $\langle \cdot,\cdot\rangle$ is an inner product in $L^{2}(D).$
\subsection{Circular cylinder case}
In this subsection we give estimate of the  first eigenvalue of \eqref{spec2} in the circular cylinder.

\begin{theorem}\label{thm12}
Suppose that $\frac{1}{2}<\alpha\leq 1$ and $s\in(0,1)$. Then first eigenvalue of \eqref{spec2} is minimised in the circular cylinder $\mathcal{C}$ among all cylindric domains of a given measure, that is
\begin{equation}
   \nu_{1}(D)\geq\nu_{1}(\mathcal{C}),
\end{equation}
for all $D$  with $|D| = |\mathcal{C}|$.
\end{theorem}
\begin{proof}
Recall that $D=(a,b) \times \Omega$ is a bounded measurable set in $\mathbb {R}^{N+1}$. Its symmetric
rearrangement $\mathcal{C}=(a,b)\times B$ is the circular cylinder  with the measure equal to the measure of $D$, i.e. $|D| = |\mathcal{C}|$. Here $B\subset \mathbb{R}^N$ is an open ball. Let $u$ be a nonnegative measurable function in $D$, such that all its
positive level sets have finite measure. With the definition of the symmetric-decreasing
rearrangement of $u$ we can use the layer-cake decomposition \cite{LL}, which expresses a
nonnegative function $u$ in terms of its level sets as
\begin{equation}
u(x,y)=\int^{\infty}_{0}\chi_{\{u(x,y)>z\}}dz,\,\,\,\forall y \in \Omega,
\end{equation}
where $\chi$ is the characteristic function of the domain. The function
\begin{equation}
u^{*}(x,y)=\int^{\infty}_{0}\chi_{\{u(x,y)>z\}^{*}}dz,\,\,\,\forall y \in \Omega,
\end{equation}
is called the (radially) symmetric-decreasing rearrangement of a nonnegative measurable
function $u$.

If a domain $D$ is the cylindrical domain, we can use Fourier's method, hence we have $u(x,y)=X(x)\varphi(y)$ and $u_{1}(x,y)=X_{1}(x)\varphi_{1}(y)$ is the first eigenfunction of the operator \eqref{FrLap}- \eqref{DirCon2}. Then we have,
\begin{equation}\label{raz}
\varphi_{1}(y)D_{a+,x}^{\alpha} \mathcal{D}_{b-,x}^{\alpha}X_{1}(x)+X_{1}(x)(-\Delta)^{s}_{y}\varphi_{1}(y)=\nu_{1}X_{1}(x)\varphi_{1}(y).\end{equation}
Let us denote
$$((-\Delta)_{y}^{s}g,g)_{\Omega}=\left(\int_{\Omega}\int_{\Omega}\frac{|g(y)-g(t)|^{2}}{|y-t|^{N+2s}}dtdy\right)^{\frac{1}{2}}.$$
By the variational principle for the self-adjoint  positive operator $\mathcal{L}^{\alpha,s}$, we get
\begin{align*}\nu_{1}(D)&=\frac{\int_{a}^{b}X_{1}(x)D_{a+,x}^{\alpha} \mathcal{D}_{b-,x}^{\alpha}X_{1}(x)dx\int_{\Omega}\varphi_{1}^{2}(y)dy+(\int_{a}^{b}X^{2}_{1}(x)dx)((-\Delta)_{y}^{s}\varphi_{1},\varphi_{1})^{2}_{\Omega} }{\int_{a}^{b}X^{2}_{1}(x)dx\int_{\Omega}\varphi^{2}_{1}(y)dy}\\& =\frac{\int_{a}^{b}X_{1}(x)D_{a+,x}^{\alpha} \mathcal{D}_{b-,x}^{\alpha}X_{1}(x)dx\int_{\Omega}\varphi_{1}^{2}(y)dy+\lambda_{1}(\Omega)\int_{a}^{b}X^{2}_{1}(x)dx\int_{\Omega}\varphi^{2}_{1}(y)dy }{\int_{a}^{b}X^{2}_{1}(x)dx\int_{\Omega}\varphi^{2}_{1}(y)dy}\\&=\frac{\int_{a}^{b}X_{1}(x)D_{a+,x}^{\alpha} \mathcal{D}_{b-,x}^{\alpha}X_{1}(x)dx\int_{\Omega}\varphi_{1}^{2}(y)dy+\lambda_{1}(\Omega)\int_{a}^{b}X^{2}_{1}(x)dx\int_{\Omega}\varphi^{2}_{1}(y)dy }{\int_{a}^{b}X^{2}_{1}(x)dx\int_{\Omega}\varphi^{2}_{1}(y)dy},\end{align*}
where $\lambda_{1}(\Omega)$ is the first eigenvalue of the fractional Dirichlet-Laplacian \eqref{spec}.

For each non-negative function $v\in L^{2}(\Omega),$  we obtain
\begin{equation}\label{rav1}
\int_{\Omega}|v(y)|^{2}dy=\int_{B}|v^{*}(y)|^{2}dy.
\end{equation}
By using Theorem A.1 in \cite{FS} and \eqref{rav1}, we establish
\begin{align*}\nu_{1}(D)&=\frac{\int_{a}^{b}X_{1}(x)D_{a+,x}^{\alpha} \mathcal{D}_{b-,x}^{\alpha}X_{1}(x)dx\int_{\Omega}\varphi_{1}^{2}(y)dy+\lambda_{1}(\Omega)\int_{a}^{b}X^{2}_{1}(x)dx\int_{\Omega}\varphi^{2}_{1}(y)dy }{\int_{a}^{b}X^{2}_{1}(x)dx\int_{\Omega}\varphi^{2}_{1}(y)dy}\\&\stackrel{\eqref{FEi}}\geq
\frac{\int_{a}^{b}X_{1}(x)D_{a+,x}^{\alpha} \mathcal{D}_{b-,x}^{\alpha}X_{1}(x)dx\int_{B}(\varphi^{*}_{1}(y))^{2}dy+\lambda_{1}(B)\int_{a}^{b}X^{2}_{1}(x)dx\int_{B}(\varphi^{*}_{1}(y))^{2}dy }{\int_{a}^{b}X^{2}_{1}(x)dx\int_{B}(\varphi^{*}_{1}(y))^{2}dy}\\&=\frac{\int_{a}^{b}X_{1}(x)D_{a+,x}^{\alpha} \mathcal{D}_{b-,x}^{\alpha}X_{1}(x)dx\int_{B}(\varphi^{*}_{1}(y))^{2}dy+\int_{a}^{b}X^{2}_{1}(x)dx((-\Delta)_{y}^{s}\varphi^{*}_{1},\varphi^{*}_{1})^{2}_{B} }{\int_{a}^{b}X^{2}_{1}(x)dx\int_{B}(\varphi^{*}_{1}(y))^{2}dy}\\&
\geq \inf_{u^*_1(x,y)\neq 0}\frac{\langle \mathcal{L}^{\alpha,s}u_1^*,u_1^*\rangle}{\|u_1^*\|^{2}_{L^{2}(\mathcal{C})}}=\nu_{1}(\mathcal{C}).
\end{align*}
The proof is complete.
\end{proof}

\begin{cor}
Suppose that $\frac{1}{2}<\alpha\leq 1$ and $s\in(0,1)$. Then first characteristic number $\mu_1(D)=\frac{1}{\nu_1(D)}$ of \eqref{spec2} is maximised in the circular cylinder $\mathcal{C}$ among all cylindric domains of a given measure, that is
\begin{equation}
   \mu_{1}(D)\leq\mu_{1}(\mathcal{C}),
\end{equation}
for all $D$  with $|D| = |\mathcal{C}|$.
\end{cor}
\subsection{Polygonal cylindric case} In this subsection we show the Rayleigh-Faber-Krahn inequality on the triangular and quadrilateral cylinders. Firstly, we recall the definition of the Steiner symmetrization (see \cite{Hen} and \cite{C17}).

Let $u(x,y)$ be a nonnegative, measurable function on $(a,b)\times\mathbb{R}^{N}$, and let $V$ be a $N-1$ dimensional plane through the origin of $\mathbb{R}^{N}$. Choose an orthogonal coordinate system in $\mathbb{R}^{N}$ such that the $y^{1}$-axis is perpendicular to $V\ni z=(y^{2},\ldots,y^{N})$.
\begin{definition}
A nonnegative, measurable function $u^{\star}(x,y)$ on $(a,b)\times\mathbb{R}^{n}$ is called the Steiner symmetrization with respect to $V$ of the function $u(x,y)$, if $u^{\star}(x,y^{1},y^{2},\ldots,y^{N})$ is a symmetric decreasing rearrangement with respect to $y^{1}$ of $u(x,y^{1},y^{2},\ldots,y^{N})$ for each fixed $y^{2},\ldots,y^{N}$.
\end{definition}

The Steiner symmetrization (with respect to the $y^{1}$-axis) $D^{\star}=(a,b)\times \Omega^{\star}$
of a measurable set $D=(a,b)\times\Omega$ is defined in the following way:
if we write $y=(y^{1},z)$ with $z\in \mathbb{R}^{N-1}$, and let
$\Omega_{z}=\{y^{1}: (y^{1},z)\in\Omega\}$, then
$$
D^{\star}:=\{(x,y^{1},z)\in (a,b)\times\mathbb{R}\times\mathbb{R}^{N-1}: y^{1}\in \Omega^{*}_{z}\},
$$
where $\Omega^{*}_{z}$ is the symmetric rearrangement of $\Omega_{z}$
(see the proof of Theorem \ref{thm12}).
Then, we have the Rayleigh-Faber-Krahn inequality on triangle and quadrilateral cylinders.
\begin{theorem}[\cite{C17}, Theorem 1.1]
The equilateral triangle has the least first eigenvalue for the fractional Dirichlet $p$-Laplacian among all triangles of given measure. The square has the
least first eigenvalue for the fractional Dirichlet $p$-Laplacian among all quadrilaterals of given measure. Moreover, the equilateral triangle and the square are the unique
minimizers in the above problems.
\end{theorem}
Let us consider the following eigenvalue problem in triangular (or quadrilateral) cylindrical domain:
\begin{equation}\label{spec3}
\mathcal{L}^{\alpha,s}u(x,y)=\nu u(x,y),\,\,\textrm{in}\,\,D=(a,b)\times\Omega,
\end{equation} with Dirichlet boundary conditions \eqref{DirCon1}-\eqref{DirCon2},
where $\Omega$ is a triangle (or quadrilateral).
Let us give the main result of this subsection.
\begin{theorem}\label{thm13}
Suppose that $\frac{1}{2}<\alpha\leq 1$ and $s\in(0,1)$. Then first eigenvalue of the \eqref{spec3} is minimised in the equilateral triangular (or square) cylinder $D^{\star}=(a,b)\times\Omega^{\star}$ among all triangular (or quadrilateral) cylindric domains of a given measure, that is
\begin{equation}
   \nu_{1}(D)\geq\nu_{1}(D^{\star}),
\end{equation}
for all $D$  with $|D| = |D^{\star}|$.
\end{theorem}
\begin{proof}
Since the Steiner symmetrization has the same property as the symmetric-decreasing rearrangement, proof of this theorem is similar to Theorem \ref{thm12}, but instead of the symmetric-decreasing rearrangement we use the Steiner symmetrization.
\end{proof}
\section{Lyapunov and Hartmann-Wintner inequalities}
\subsection{Lyapunov inequality}
\label{SEC:2}
Let us consider the fractional elliptic equation:
 \begin{equation}\label{osn}
\mathcal{L}^{\alpha,s}u(x,y)=q(x)u(x,y),\,\,\textrm{in}\,\,D= (a,b)\times\Omega,
\end{equation}
with boundary conditions \eqref{DirCon1}-\eqref{DirCon2}, where $q(x)$ be a real-valued, continuous function.

In this section we show a Lyapunov-type inequality for \eqref{osn}.
\begin{theorem}\label{thm2}
Assume that $\frac{1}{2}<\alpha\leq 1$ , $s\in(0,1)$ and $q\in C([a,b])$. Then for \eqref{osn}, we get

\begin{equation}
\int_{a}^{b}|q(x)-\lambda_{1}(\Omega)|dx\geq\left(\sup\limits_{a<x<b}G(x,x)\right)^{-1},
\end{equation}
where $\lambda_{1}(\Omega)$ is the first eigenvalue of \eqref{spec}.
\end{theorem}
\begin{proof}
By multiplying \eqref{osn} with $\varphi_{1}(y)$ and integrating over $\Omega$, we obtain
\begin{align*}
&\int_{\Omega}D_{a+,x}^{\alpha} \mathcal{D}_{b-,x}^{\alpha}u(x,y)\varphi_{1}(y)dy+\int_{\Omega}((-\Delta_{y})^{s}u(x,y))\varphi_{1}(y)dy -q(x)\int_{\Omega}u(x,y)\varphi_{1}(y)dy\\
=&D_{a+,x}^{\alpha} \mathcal{D}_{b-,x}^{\alpha}\int_{\Omega}u(x,y)\varphi_{1}(y)dy +\int_{\Omega}((-\Delta_{y})^{s}u(x,y))\varphi_{1}(y)dy-q(x)\int_{\Omega}u(x,y)\varphi_{1}(y)dy\\
=&D_{a+,x}^{\alpha} \mathcal{D}_{b-,x}^{\alpha}\int_{\Omega}u(x,y)\varphi_{1}(y)dy +\int_{\Omega}((-\Delta_{y})^{s}\varphi_{1}(y))u(x,y)dy -q(x)\int_{\Omega}u(x,y)\varphi_{1}(y)dy\\
=&D_{a+,x}^{\alpha} \mathcal{D}_{b-,x}^{\alpha}\int_{\Omega}u(x,y)\varphi_{1}(y)dy +\lambda_{1}(\Omega)\int_{\Omega}u(x,y)\varphi_{1}(y)dy-q(x)\int_{\Omega}u(x,y)\varphi_{1}(y)dy
\\
=&D_{a+,x}^{\alpha} \mathcal{D}_{b-,x}^{\alpha}v(x)-q_{1}(x)v(x)=0,
\end{align*}
where $v(x)=\int_{\Omega}u(x,y)\varphi_{1}(y)dy,$ $q_{1}(x)=q(x)-\lambda_{1}(\Omega),$ from boundary conditions \eqref{DirCon1}, \eqref{DirCon2}, we get
$$v(a)=0,\, v(b)=0.$$ That is
$$D_{a+,x}^{\alpha} \mathcal{D}_{b-,x}^{\alpha}v(x)-q_{1}(x)v(x)=0,\,x\in(a,b),\,\,\,v(a)=0,\,\, v(b)=0.$$

By using Theorem \ref{thm1} and Lemma \ref{lem1}, we establish
\begin{equation}
    |u(x)|\leq \int_{a}^{b}|G(x,t)||q_1(t)|u(t)|dt\stackrel{\eqref{eq1}}\leq  G(x,x)\left(\sup\limits_{a<t<b}|u(t)|\right)\int_{a}^{b}|q_1(t)|dt.
\end{equation}
Then by taking supremum in both sides in $a<x<b$, we have
\begin{equation}\label{Sup}
   \left(\sup\limits_{a<x<b}G(x,x)\right)^{-1}\leq\int_{a}^{b}|q_1(t)|dt.
\end{equation}
Finally, by using \eqref{Sup}, we have
\begin{equation}\begin{split}
\int_{a}^{b}|q_{1}(x)|dx&=\int_{a}^{b}|q(x)-\lambda_{1}(\Omega)|dx \geq\left(\sup\limits_{a<x<b}G(x,x)\right)^{-1}.
\end{split}\end{equation}
The proof of Theorem \ref{thm2} is complete.
\end{proof}
\begin{cor}
By taking $\alpha=1$, we get
$$\int_{a}^{b}|q_{1}(x)|dx=\int_{a}^{b}|q(x)-\lambda_{1}(\Omega)|dx \geq\frac{4}{b-a}.$$
\end{cor}
\begin{proof}
Firstly, let us calculate some integrals. Then,
\begin{equation}
    K(x,x)=K(x,a)=b-x,
\end{equation}
and
\begin{equation}
    K(a,a)=b-a.
\end{equation}
Then by using these facts, we have
$$G(x,x)=b-x-\frac{(b-x)^{2}}{b-a}.$$
Then supremum of the function $G(x,x)$ on $a<x<b$ equals to $\frac{b-a}{4}$.
\end{proof}

\begin{theorem}
Suppose that $\frac{1}{2}<\alpha<1$ and $s\in(0,1)$. Then we have,
\begin{equation}\begin{split}
\int\limits^b_a|q(x)|dx+(b-a)\lambda_{1}(\Omega)&\geq \int\limits^b_a|q(x)|dx+(b-a)\lambda_{1}(B)\\& \geq\left(\sup\limits_{a<x<b}G(x,x)\right)^{-1},
\end{split}\end{equation}
where $\lambda_{1}(B)$ is the first eigenvalue of the eigenvalue problem \eqref{spec2} in a ball $B$ with $|\Omega|=|B|$.
\end{theorem}

\begin{proof}
By the previous Theorem, let $B$ be a ball, then by using Theorem A.1 in \cite{FS}, we have
\begin{equation}\begin{split}
\int\limits^b_a|q(x)|dx+(b-a)\lambda_{1}(\Omega)&\geq
\int\limits^b_a|q(x)|dx+(b-a)\lambda_{1}(B)\\& \geq
\int_{a}^{b}|q(x)-\lambda_{1}(B)|dx \\&
\geq \left(\sup\limits_{a<x<b}G(x,x)\right)^{-1},
\end{split}\end{equation} completing the proof.
\end{proof}

\subsection{Hartman-Wintner inequality} In this section, we show a Hartman-Wintner type inequality for problem \eqref{osn}, \eqref{DirCon1}, \eqref{DirCon2}.
\begin{theorem}
Let $\frac{1}{2}<\alpha\leq1$ and $s\in(0,1)$,  and $q\in C([a,b])$. Suppose that the fractional boundary value problem \eqref{osn}, \eqref{DirCon1}, \eqref{DirCon2} has a nontrivial continuous solution. Then, we have
\begin{equation}\label{hw}
\int\limits_a^b \left(K(a,a)K(s,s)-K^{2}(a,s)\right) [q(x)-\lambda_{1}(\Omega)]^+ ds\geq \frac{(b-a)^{2\alpha-1}}{\Gamma^{2}(\alpha)(2\alpha-1)},
\end{equation}
where $[q(x)-\lambda_{1}(\Omega)]^+ = \max\{q(x)-\lambda_{1}(\Omega), 0\}.$
\end{theorem}
\begin{proof}
By multiplying \eqref{osn} with $\varphi_{1}(y)$ and integrating over $\Omega,$ for the function $v(x)=\int_{\Omega}u(x,y)\varphi_{1}(y)dy$ we have the following problem:
\begin{equation}\label{lyap1}
 \begin{cases}
D_{a+,x}^{\alpha} \mathcal{D}_{b-,x}^{\alpha}v(x)-q_{1}(x)v(x)=0,\,\,\,\,x\in(a,b),\\
   v(a)=0,\, v(b)=0.
 \end{cases}
\end{equation}
By Theorem \ref{thm1} the problem \eqref{lyap1} is equivalent to the integral equation  $$v(x)=\int\limits_a^b G(x,s) q_1(s)v(s)ds,$$ where $$G(x,t)=\frac{K(x,t)}{\Gamma^{2}(\alpha)}-\frac{K(a,t)K(x,a)}{\Gamma^{2}(\alpha)K(a,a)},$$

$$K(x,t)=\frac{1}{\Gamma^{2}(\alpha)}\int_{\max\{x,t\}}^{b}(s-x)^{\alpha-1}(s-t)^{\alpha-1}ds,$$
and from Lemma \ref{lem1}, we have
\begin{equation}\label{green1}
 G(x,s)\leq G(s,s),\,\, \textrm{for} \,\, a<x<s<b.
\end{equation}
From this, by \eqref{green1} for any $a \leq x \leq b,$ we obtain
\begin{align*}|v(x)|&\leq\int\limits_a^b |G(x,s)| |q_1(s)||v(s)|ds\\& \leq \int\limits_a^b G(s,s) |q_1(s)||v(s)|ds\\&
=\frac{1}{K(a,a)}\int\limits_a^b \left(K(a,a)K(s,s)-K^{2}(a,s)\right) |q_1(s)||v(s)|ds\\&
=\frac{\Gamma^{2}(\alpha)(2\alpha-1)}{(b-a)^{2\alpha-1}}\int\limits_a^b\left(K(a,a)K(s,s)-K^{2}(a,s)\right) |q_1(s)||v(s)|ds,\end{align*} thanks to $K(a,a)=\frac{(b-a)^{2\alpha-1}}{\Gamma^2(\alpha)(2\alpha-1)}.$
Theorem \ref{hw} is proved.
\end{proof}
\begin{cor}
By taking $\alpha=1$ and $s=1$ in \eqref{hw}, we get the classical Hartman-Wintner inequality
\begin{equation}
\int_{a}^{b}(b-s)(s-a)q_1^{+}(s)\geq b-a.
\end{equation}
\end{cor}

\end{document}